\documentclass[11pt]{amsart}
\usepackage{amsfonts,latexsym,amsthm,amssymb,graphicx}
\usepackage[all]{xy}
\usepackage{hyperref}
\usepackage[usenames]{color} 

\DeclareFontFamily{OT1}{rsfs}{}
\DeclareFontShape{OT1}{rsfs}{n}{it}{<-> rsfs10}{}
\DeclareMathAlphabet{\mathscr}{OT1}{rsfs}{n}{it}

\CompileMatrices

\newtheorem{theorem}{Theorem}[section]
\newtheorem{lemma}[theorem]{Lemma}
\newtheorem{claim}[theorem]{Claim}
\newtheorem{conj}{Conjecture}

{\theoremstyle{definition} \newtheorem{defin}[theorem]{Definition}}
{\theoremstyle{remark} \newtheorem{remark}[theorem]{Remark}
\newtheorem{example}[theorem]{Example}}

\numberwithin{equation}{section}

\newcommand{\Abb}{{\mathbb{A}}}
\newcommand{\Cbb}{{\mathbb{C}}}
\newcommand{\Lbb}{{\mathbb{L}}}
\newcommand{\Pbb}{{\mathbb{P}}}
\newcommand{\Qbb}{{\mathbb{Q}}}
\newcommand{\Zbb}{{\mathbb{Z}}}
\newcommand{\Til}[1]{{\widetilde{#1}}}
\newcommand{\saf}{\,}
\newcommand{\caM}{{\mathcal{M}}}
\newcommand{\ocaM}{{\overline\caM}}
\DeclareMathOperator{\rk}{rk}
\newcommand{\Var}{\mathrm{Var}}
\newcommand{\qede}{\hfill $\lrcorner$}

\title{Log concavity of the Grothendieck class of $\overline{\mathcal M}_{0,n}$}
\author{Paolo Aluffi}
\author{Stephanie Chen}
\author{Matilde Marcolli}
\address{
Mathematics Department, 
Florida State University,
Tallahassee FL 32306, U.S.A.
}
\address{
Department of Mathematics, 
California Institute of Technology,
Pasadena CA 91105, U.S.A.
}
\email{aluffi@math.fsu.edu}
\email{schen7@caltech.edu}
\email{matilde@caltech.edu}

\begin{document}

\begin{abstract}
Using a known recursive formula for the Grothendieck classes of the moduli 
spaces~$\overline{\mathcal M}_{0,n}$, 
we prove that they satisfy an asymptotic form of ultra-log-concavity as polynomials in the
Lefschetz class. We also observe that these polynomials are $\gamma$-positive.
Both properties, along with numerical evidence, support the conjecture that 
these polynomials only have real zeros. This conjecture may be viewed as a particular case 
of a possible extension of a conjecture of Ferroni-Schr\"oter and Huh on Hilbert series of 
Chow rings of matroids.

We prove asymptotic ultra-log-concavity by studying differential equations obtained
from the recursion, whose solutions are the generating functions of the individual betti
numbers of $\overline{\mathcal M}_{0,n}$. We obtain a rather complete description of these generating
functions, determining their asymptotic behavior; their dominant term is controlled by the 
coefficients of the Lambert W function.
The $\gamma$-positivity property follows directly from the recursion, extending the 
argument of Ferroni et al.~proving $\gamma$-positivity for the Hilbert series of the Chow ring
of matroids.
\end{abstract}

\maketitle

{\let\thefootnote\relax\footnotetext{2024/02/07}}

%%%

\section{Introduction}\label{sec:intro}
As a straightforward consequence of the Hard Lefschetz theorem, the sequence
of (even) betti numbers of a smooth complex projective variety is {\em unimodal.\/}
This fact is discussed in detail in~\cite[Theorem~18]{MR1110850}. 
The sequence is not necessarily {\em log-concave,\/} but there are situations where
it is expected to be; for example, this is discussed in~\cite{MR4565658} for the
case of configuration spaces, providing log-concavity results for e.g., the space of 
ordered $n$-uples of points in $\Cbb$.

The object of study of this note is the moduli space 
$\ocaM_{0,n}$ of stable $n$-pointed rational curves for $n\ge 3$.
We prove an asymptotic log-concavity property for the Poincar\'e polynomials of
these varieties. We also remark that these polynomials are `$\gamma$-positive'. 
These results may be viewed as evidence for a conjecture stating that the Poincar\'e
polynomials only have real zeros, see below.\smallskip

We focus on the class of $\ocaM_{0,n}$ in the Grothendieck group of varieties
$K(\Var_\Cbb)$.
This is a universal Euler characteristic, therefore {\em a priori\/} a more fundamental object. 
It is known (cf.~\cite{MR3701904}, and~\S\ref{sec:GcM0n} below) that the class 
of~$\ocaM_{0,n}$ is a polynomial with integer coefficients in the Lefschetz-Tate
class $\Lbb=[\Abb^1]$; we denote this class by
\[
[\ocaM_{0,n}] = a_{n,0} + a_{n,1} \Lbb + \cdots + a_{n,n-3} \Lbb^{n-3}\saf.
\]
The Poincar\'e polynomial is given by specializing $\Lbb$ to $t^2$. Thus, $\ocaM_{0,n}$
only has even cohomology (also cf.~\cite[p.~549]{MR1034665}), and the integers $a_{n,k}$ may 
be interpreted as the ranks of the cohomology groups of $\ocaM_{0,n}$. Log-concavity of 
these polynomials amounts to the statement that $a_{n,i}^2 \ge a_{n,i-1} a_{n,i+1}$ for 
all $i\ge 1$ and all $n\ge 3$. The stronger condition of {\em ultra-log-concavity\/} is the 
inequality
\[
\left(\frac{a_{n,i}}{\binom {n-3}i}\right)^2 \ge \frac{a_{n,i-1}}{\binom {n-3}{i-1}}\cdot 
\frac{a_{n,i+1}}{\binom {n-3}{i+1}}\saf.
\]
for all $i\ge 1$, all $n\ge 3$.

\begin{theorem}\label{thm:main}
With notation as above, $\forall i\ge 1$ $\exists N$ s.t.~$\forall n\ge N$ 
\begin{equation}
\left(\frac{a_{n,i}}{\binom {n-3}i}\right)^2 \ge \frac{a_{n,i-1}}{\binom {n-3}{i-1}}\cdot 
\frac{a_{n,i+1}}{\binom {n-3}{i+1}}\saf.
\end{equation}
\end{theorem}

Thus, an asymptotic log-concavity property holds for the coefficients of the Grothendieck
class of $\ocaM_{0,n}$, hence for the betti numbers $a_{n,k}=\rk H^{2k}(\ocaM_{0,n},\Qbb)$. 

The class $[\ocaM_{0,n}]$ is explicitly known recursively, see~\cite{MR1034665}, 
\cite[Proposition~3.2]{MR3701904}, and~\eqref{eq:recu} below. The first several expressions 
for this class are
\begin{gather*}
1 \\
\Lbb+1 \\
\Lbb^2 + 5\Lbb + 1 \\
\Lbb^3 + 16 \Lbb^2 +16 \Lbb + 1 \\
\Lbb^4 + 42\Lbb^3 + 127 \Lbb^2 +42 \Lbb + 1 \\
\Lbb^5 + 99\Lbb^4 + 715\Lbb^3 + 715 \Lbb^2 +99 \Lbb + 1 \\
\Lbb^6 + 219\Lbb^5 + 3292\Lbb^4 + 7723\Lbb^3 + 3292 \Lbb^2 + 219 \Lbb + 1 \\
\Lbb^7 + 466\Lbb^6 + 13333\Lbb^5 + 63173\Lbb^4 + 63173\Lbb^3 + 13333 \Lbb^2 
+ 466 \Lbb + 1 \\
\Lbb^8 + 968\Lbb^7 + 49556\Lbb^6 + 429594\Lbb^5 + 861235\Lbb^4 + 429594\Lbb^3 
+ 49556 \Lbb^2 + 968 \Lbb + 1
\end{gather*}

Numerical evidence supports the following conjecture.
\begin{conj}\label{conj:rr}
The polynomial $P_n(t)\in \Zbb[t]$, such that $[\ocaM_{0,n}]=P_n(\Lbb)$, has only real
zeros.
\end{conj}

Due to a standard result attributed to Newton (\cite[Theorem~2]{MR1110850}), this 
conjecture would imply ultra-log-concavity of the polynomials. Thus, Theorem~\ref{thm:main}
gives some support to Conjecture~\ref{conj:rr}.

Related real-rootedness conjectures are listed in~\cite[Conjecture~1.6]{FMSV}. 
Specifically, the first of these conjectures, due independently to Ferroni-Schr\"oter 
and Huh, posits that the Hilbert series of the Chow ring of an arbitrary matroid should 
only have real roots; see~\cite[Conjecture 8.18]{FerrSchr}. The conventional definition for 
the Chow ring of a matroid is given with respect to its {\em maximal\/} building set; using the 
{\em minimal\/} rather than the maximal building set, the Chow ring of the braid matroid
agrees with the cohomology of $\ocaM_{0,n}$, see~\S5.2 in {\em loc.~cit.}
Thus, Conjecture~\ref{conj:rr} addresses a particular case of a possible extension of
\cite[Conjecture~1.6]{FMSV}; but we note that the Hilbert series of the Chow ring of a matroid 
with respect to the minimal building set is in general not real-rooted.

Poincar\'e duality implies that the polynomials expressing $[\ocaM_{0,n}]$ are palindromic.
For palindromic polynomials with nonnegative coefficients, real-rootedness also implies
`$\gamma$-{posi\-tivity}', which amounts to the positivity of the coefficients of the polynomials
in a basis consisting of polynomials of the type $t^i (1+t)^j$ (see~\S\ref{sec:gapol} for
the precise definition). The following result is a straightforward consequence of the recursive
formula~\eqref{eq:recu} determining~$[\ocaM_{0,n}]$, and may be viewed as further
evidence for Conjecture~\ref{conj:rr}.

\begin{theorem}\label{thm:gp}
For all $n\ge 3$, the polynomial $P_n(t)\in \Zbb[t]$ such that $[\ocaM_{0,n}]=P_n(\Lbb)$
is $\gamma$-positive.
\end{theorem}
\smallskip

The recursive formula determining $[\ocaM_{0,n}]$ is proved in~\cite{MR3701904} by an
argument using a suitable tree-level partition function. This method is modeled on the 
analogous result for the Poincar\'e polynomial of $\ocaM_{0,n}$ obtained by Y.~Manin
in~\cite{MR1363064}. The recursion is equivalent to a recursion for the betti numbers 
stated by S.~Keel in~\cite[p.~550]{MR1034665}, following from his complete determination
of the Chow groups of $\ocaM_{0,n}$. For the convenience of the reader, in~\S\ref{sec:GcM0n}
we reprove the recursion formula for the Grothendieck class of~$\ocaM_{0,n}$ directly 
from Keel's description of $\ocaM_{0,n}$ as a sequence of blow-ups over 
$\ocaM_{0,n-1}\times \ocaM_{0,4}$. In~\S\ref{sec:gapol} we prove Theorem~\ref{thm:gp}
as a direct consequence of the recursive formula~\eqref{eq:recu}.

Keel's recursion involves the whole set of betti numbers, while in order to prove 
Theorem~\ref{thm:main} it is necessary to have information about individual betti
numbers $a_{n,k}$ for fixed~$k$. In~\S\ref{sec:difeq} we explain how to obtain 
first order linear differential equations satisfied by the generating functions 
$\alpha_k(z):=\sum_{n\ge 3} a_{k,n} \frac{z^{n-1}}{(n-1)!}$, determining these 
functions along with the initial condition $\alpha_k(0)=0$. For instance, 
\[
\frac{d\alpha_2}{dz}= \alpha_2(z) +3 e^{3z} - \frac{2z^2+10z+10}2 e^{2z} + 
\frac{z^3+5z^2+8z+4}2 e^z\saf,
\]
from which
\begin{align*}
\alpha_2(z) &= 
\frac{3  e^{3z}}2 
- (z+1)(z + 2)e^{2z}  
+ \left( \frac{z^4}8 + \frac{5 z^3}6 + 2 z^2+2  z + \frac 12\right) e^z \\
&= 1\cdot \frac{z^4}{4!}+16 \cdot \frac{z^5}{5!}+ 127\cdot \frac{z^6}{6!}+
715\cdot \frac{z^7}{7!}+ 3292 \cdot \frac{z^8}{8!}+ 13333\cdot \frac{z^9}{9!}+
49556\cdot \frac{z^{10}}{10!}+\cdots
\end{align*}
recovering the coefficients of $\Lbb^2$ in the table displayed above.

We use an inductive argument to obtain a general description of these
generating functions as a combination of exponentials with (signed) polynomial 
coefficients $p_m^{(k)}(z)\in \Qbb[z]$
(Theorem~\ref{thm:main2sharp}). These polynomials certainly deserve further study;
we establish their degree and that their leading coefficient is positive, and conjecture
that they are ultra-log-concave.

The dominant terms in the expressions we obtain determine the asymptotic behavior 
of~$\alpha_{k,n}$.

\begin{theorem}\label{thm:main2}
For all $k\ge 0$,
\[
\alpha_{k,n} = \rk H^{2k}(\ocaM_{0,n}) \sim \frac{(k+1)^{k+n-1}}{(k+1)!}
\]
as $n\to \infty$.
\end{theorem}

We remark here that Theorem~\ref{thm:main2} is equivalent to the statement that, 
as $n\to \infty$,
\[
-\sum_{k=0}^{n-3} \frac{\rk H^{2k}(\ocaM_{0,n})}{(k+1)^{n-1}}\, (-t)^{k+1} \sim W(t)
\]
(in the sense that for every $k\ge 0$, the coefficient of $t^k$ in the l.h.s.~converges to
the corresponding coefficient in the r.h.s.) where $W(t)$ is the {\em principal branch of 
the Lambert~W function,\/} characterized by the identity $W(t) e^{W(t)}=t$.\smallskip

Theorem~\ref{thm:main} follows easily from Theorem~\ref{thm:main2}, see~\S\ref{sec:proof}.
In fact, in~\S\ref{sec:proof} we will obtain a more precise result than 
Theorem~\ref{thm:main2}. We will show that there exist polynomials
$q_m^{(k)}(n)\in \Qbb[n]$ of degree $2m$, with positive leading coefficient, such that
\[
\rk H^{2k}(\ocaM_{0,n})=\frac{(k+1)^{k-1}}{(k+1)!}\cdot (k+1)^n
+\sum_{m=1}^k (-1)^m q^{(k)}_m(n)\cdot (k+1-m)^n
\]
(Theorem~\ref{thm:main2again}, Remark~\ref{rem:clo}).
The polynomials $q_m^{(k)}$ have straightforward expressions in terms of the coefficients
of the polynomials $p_m^{(k)}$ mentioned above, and are also object of evident interest.

\medskip

{\em Acknowledgments.} 
The authors are grateful to J.~Huh for pointing out reference~\cite{FMSV} and to Luis
Ferroni, Matt Larson, and Sam Payne for helpful comments.
P.A.~was supported in part by the Simons Foundation, collaboration grant~\#625561,
and by an FSU `COFRS' award. He thanks Caltech for hospitality. S.C.~was supported
by a Summer Undergraduate Research Fellowship at Caltech. M.M.~was
supported by NSF grant DMS-2104330. 

%%%

\section{Recursion for the Grothendieck class of $\ocaM_{0,n}$}\label{sec:GcM0n}

The class $[\ocaM_{0,n}]$ in the Grothendieck group $K(\Var_k)$, $k$ any algebraically
closed field, is determined by the following recursion.

\begin{theorem}\label{thm:MM}
$[\ocaM_{0,3}]=1$. For $n>3$, 
\begin{equation}\label{eq:recu}
[\ocaM_{0,n}] = [\ocaM_{0,n-1}] (1+\Lbb)
+\Lbb\, \sum_{i=3}^{n-2} \binom {n-2}{i-1} [\ocaM_{0,i}] 
[\ocaM_{0,n+1-i}]\saf.
\end{equation}
\end{theorem}

This formula is equivalent to the statement given in~\cite{MR3701904}, proved there
by the same method used to prove an analogous statement for the Poincar\'e polynomial
in~\cite{MR1363064}, that is, by adding contributions of strata of $\ocaM_{0,n}$.
Ultimately, the recursion follows from
\[
[\caM_{0,k}] 
= (\Lbb-2)\cdots (\Lbb-k+2)\saf,
\] 
which is easily proved directly, and a sum over trees performed by using (to 
quote~\cite{MR1363064}) `a general formula of perturbation theory in order to reduce 
the calculation of the relevant generating functions to the problem of finding the critical 
value of an appropriate formal potential.'

The recursion is equivalent to a recursive formula determining the set of betti numbers
of~$\ocaM_{0,n}$, given\footnote{We alert the reader to two typos in the cited formula
in~\cite{MR1034665}:
the binomial $\binom nk$ should be $\binom nj$, and the expression $n-j-1$ should 
be $n-j+1$.} in~\cite[p.~550]{MR1034665}. In this reference, the formulas for the betti numbers
are presented as a consequence of the determination of the Chow groups of $\ocaM_{0,n}$,
\cite[Theorem~1,~\S3]{MR1034665}. The literature on such formulas is very rich. We mention
that the recursion is equivalent to a functional equation obtained by Getzler as a consequence 
of~\cite[Theorem~5.9]{MR1363058} and presented as a reformulation of a computation 
of Fulton and MacPherson from~\cite{MR1259368}. An alternative version of the same 
functional equation is given by Manin in~\cite[(0.7)]{MR1363064}. Chen-Gibney-Krashen
extend these formulas to the case of pointed projective spaces and to the motivic setting, 
\cite{MR2496455}; Li obtains general motivic formulas for configuration spaces 
in~\cite{MR2595554}.

For the convenience of the reader, we offer a direct derivation of the recursion in the
Grothendieck group $K(\Var_k)$ from Keel's description of $\ocaM_{0,n}$.

\begin{proof}[Proof of Theorem~\ref{thm:MM}]
We recall Keel's recursive construction of $\ocaM_{0,n}$ from~\cite{MR1034665}.
The space $\ocaM_{0,3}$ is a point, and $\ocaM_{0,4}\cong\Pbb^1$. For $n> 4$, 
$\ocaM_{0,n}$ is constructed as a 
sequence of blow-ups over $\ocaM_{0,n-1}\times \ocaM_{0,4}$.  The centers of the 
blow-ups are all disjoint, smooth, of codimension~$2$. In fact, they are isomorphic to 
products
\[
\ocaM_{0,|T|+1} \times \ocaM_{0,|T^c|+1} \saf,
\]
where $T$ denotes a subset of $\{1,\dots, n-1\}$ such that the complement 
$T^c$ contains two of $1,2,3$. Note that each center is isomorphic to
$\ocaM_{0,i}\times \ocaM_{0,n+1-i}$, with $i=3,\dots, n-2$.

Now, if $\Til V$ is the blow-up of a variety $V$ along a regularly embedded 
center $B$ of codimension~$r$, then the Grothendieck class of $\Til V$ is
\[
[\Til V]=[V] + (\Lbb+\cdots +\Lbb^{r-1}) [B]\saf.
\]
Indeed, the exceptional divisor of the blow-up is isomorphic to the projectivization
of $N_BV$, a $\Pbb^{r-1}$-bundle over $B$. In the case we are considering, 
we are blowing up $\ocaM_{0,n-1}\times \ocaM_{0,4}$, with class
\[
[\ocaM_{0,n-1}\times \ocaM_{0,4}]
=[\ocaM_{0,n-1}\times \Pbb^1] = [\ocaM_{0,n-1}] (1+\Lbb)
\]
and each center has codimension~$r=2$; therefore
\[
[\ocaM_{0,n}]=[\ocaM_{0,n-1}] (1+\Lbb) +
\Lbb\sum_k [B_k]\saf,
\]
where the sum runs through the centers $B_k$ of the blow-ups. Thus, in order to
prove~\eqref{eq:recu}, it suffices to show that 
\[
\sum_k [B_k] = \sum_{i=3}^{n-2} \binom {n-2}{i-1} [\ocaM_{0,i}\times \ocaM_{0,n+1-i}]\saf.
\]
Since $\ocaM_{0,i}\times \ocaM_{0,n+1-i}\cong \ocaM_{0,n+1-i}\times \ocaM_{0,i}$, the
right-hand side equals
\begin{multline}\label{eq:coll}
\binom{n-2}{\frac{n-1}2} [\ocaM_{0,\frac{n+1}2}\times \ocaM_{0,\frac{n+1}2}]+
\sum_{3\le i< \frac{n+1}2} \left( \binom{n-2}{i-1}+\binom{n-2}{n-i}\right)
[\ocaM_{0,i}\times \ocaM_{0,n+1-i}] \\
=\binom{n-2}{\frac{n-1}2} [\ocaM_{0,\frac{n+1}2}\times \ocaM_{0,\frac{n+1}2}]+
\sum_{3\le i< \frac{n+1}2} \binom{n-1}{i-1} [\ocaM_{0,i}\times \ocaM_{0,n+1-i}]
\end{multline}
where the first summand only appears if $n$ is odd.

According to Keel's construction, for $n\ge 4$, a center isomorphic to $\ocaM_{0,i}
\times \ocaM_{0,n+1-i}$ is blown up for all sets $T\subseteq \{1,\dots, n-1\}$
such that
\begin{itemize}
\item 
$T^c$ contains at least two of $1,2,3$;
\item
$|T^c|=n-i$ or $|T^c|=i-1$.
\end{itemize}
For all $k$ between $2$ and $n-3$, the number of subsets $T$ such that $|T^c|=k$ 
and $T^c$ contains exactly two of $1,2,3$ is
\[
3\binom{n-4}{k-2}
\]
while the number of subsets $T$ such that $|T^c|=k$ and $T^c$ contains all of $1,2,3$ is
\[
\binom{n-4}{k-3}
\]
(in particular, $0$ if $k=2$). For $3\le i<\frac {n+1}2$, the number of centers
isomorphic to the product $\ocaM_{0,i} \times \ocaM_{0,n+1-i}$ is therefore
\[
3\binom{n-4}{n-i-2}+\binom{n-4}{n-i-3}+3\binom{n-4}{i-3}+\binom{n-4}{i-4} =\binom{n-1}{i-1}\saf.
\]
(Maybe more intrinsically, there is one such center for every subset $S\subseteq\{1,\dots, n-1\}$ 
of size $i-1$. Indeed, if $S$ is such a subset, then either $S$ or $S^c$ satisfies the
condition posed on~$T^c$ in Keel's prescription.) If $n$ is odd and $i=\frac {n+1}2$, 
the number of centers isomorphic to $\ocaM_{0,\frac{n+1}2} \times \ocaM_{0,\frac{n+1}2}$ is
\[
3\binom{n-4}{\frac {n+1}2-3}+\binom{n-4}{\frac {n+1}2-4}=\binom{n-2}{\frac{n-1}2}\saf.
\]
This is as prescribed in~\eqref{eq:coll}, concluding the verification.
\end{proof}

\begin{remark}
The distributions
of products in the sequence of centers and in the corresponding sum in~\eqref{eq:recu}
differ in general. For example, for $n=6$ the summation in~\eqref{eq:recu} expands to
\[
6\, [\ocaM_{0,3}\times \ocaM_{0,4}]
+ 4\, [\ocaM_{0,4}\times \ocaM_{0,3}]
\]
while Keel's construction prescribes blowing up along $7$ copies of 
$\ocaM_{0,3}\times \ocaM_{0,4}$ and $3$ copies of $\ocaM_{0,4}\times \ocaM_{0,3}$.

The recursion for the Poincar\'e polynomial following directly from Keel's recursion
in~\cite[p.~550]{MR1034665} gives yet a different decomposition: 
$5\, [\ocaM_{0,3}\times \ocaM_{0,4}]+ 5\, [\ocaM_{0,4}\times \ocaM_{0,3}]$.
\qede\end{remark}

%%%

\section{$\ocaM_{0,n}$ is $\gamma$-positive}\label{sec:gapol}

For a survey on $\gamma$-positivity, we refer the reader to~\cite{MR3878174};
we follow the terminology in~\cite[\S2.2]{FMSV}.
A univariate polynomial $f(t)=\sum a_i t^i$ is `symmetric', with `center' $\frac d2$, if $a_{d-i}
=a_i$ for all $i$. Every symmetric polynomial $f(t)\in \Zbb[t]$ with center $\frac d2$
can clearly be written
\begin{equation}\label{eq:sympo}
f(t) = \sum_{i=0}^{\lfloor \frac d2 \rfloor} \gamma_i \, t^i\, (t+1)^{d-2i}
\end{equation}
for unique integers $\gamma_i$, $i=0,\dots, \lfloor \frac d2 \rfloor$.

\begin{defin}\label{def:gp}
We say that a symmetric polynomial $f$ is {\em $\gamma$-positive\/} if all the integers 
$\gamma_i$ are nonnegative.
\qede\end{defin}

Our interest in this notion is due to the following well-known fact.

\begin{lemma}[\cite{MR3878174}, \S1; \cite{MR4681276}, Proposition~5.3]
Real-rooted symmetric polynomials with nonnegative coefficients are $\gamma$-positive.
\end{lemma}

Thus, $\gamma$-positivity may be taken as collateral evidence for real-rootedness.
In~\cite[Theorem~1.8]{FMSV} it is shown that the Hilbert series of the Chow ring of every 
matroid is $\gamma$-positive. We prove the analogous statement for $\ocaM_{0,n}$.

\begin{theorem}\label{thm:gpag}
For all $n\ge 3$, the polynomial $P_n(t)\in \Zbb[t]$ such that $[\ocaM_{0,n}]=P_n(\Lbb)$
is~$\gamma$-positive.
\end{theorem}

(This is Theorem~\ref{thm:gp}, stated in the introduction.)

\begin{proof}
Following~\cite{FMSV}, for a symmetric polynomial~\eqref{eq:sympo} we let
\[
\gamma(f):= \sum_{i=0}^{\lfloor \frac d2 \rfloor} \gamma_i \, t^i\saf.
\]
Thus, $f$ is $\gamma$-positive if and only if $\gamma(f)$ has nonnegative coefficients.
This operation satisfies several properties (see~\cite[Lemma~2.10]{FMSV}):
\begin{itemize}
\item[(i)] $\gamma(fg)=\gamma(f)\gamma(g)$;
\item[(ii)] $\gamma(tf)=t\gamma(f)$
\item[(iii)] $\gamma(f(1+t))=\gamma(f)$
\item[(iv)] If $f$ and $g$ have the same center of symmetry, then $\gamma(f+g)=\gamma(f)
+\gamma(g)$.
\end{itemize}

With this understood, the proof of Theorem~\ref{thm:gpag} is
a straightforward consequence of the recursion~\eqref{eq:recu} for the Grothendieck
class $[\ocaM_{0,n}]$. In terms of the polynomial $P_n$, this recursion reads
\[
P_n(t) = P_{n-1}(t)  (1+t)
+t \sum_{i=3}^{n-2} \binom {n-2}{i-1} P_i(t) P_{n+1-i}(t)\saf.
\]
The constant $P_3(t)=1$ is trivially $\gamma$-positive. Arguing by strong induction,
assume that $P_k(t)$ is $\gamma$-positive for all $k<n$.
The degree of $P_k(t)$ is $k-3$ and the polynomial is palindromic, so it is symmetric
with center $\frac {k-3}2$. It follows that each term $P_i P_{n+1-i}$ is symmetric with
center $\frac{n-5}2$, and $\gamma(P_i P_{n+1-i})=\gamma(P_i)\gamma(P_{n+1-i})$ 
by~(i). By (ii) and (iv),
\[
\gamma\left(t\sum_{i=3}^{n-2} \binom {n-2}{i-1} P_i(t) P_{n+1-i}(t)\right)
=t\sum_{i=3}^{n-2} \binom {n-2}{i-1} \gamma(P_i)\gamma(P_{n+1-i})\saf,
\]
and this polynomial has center $\frac{n-5}2+1=\frac{n-3}2$.
By~(iii),
\[
\gamma(P_{n-1}(t) (1+t))=\gamma(P_{n-1})\saf,
\]
and $P_{n-1}(t) (1+t)$ also has center $\frac{n-3}2$. By (iv) again, we can conclude
\begin{equation}\label{eq:recuga}
\gamma(P_n) = \gamma(P_{n-1})+t\sum_{i=3}^{n-2} \binom {n-2}{i-1} 
\gamma(P_i)\gamma(P_{n+1-i})\saf.
\end{equation}
By induction the r.h.s.~has nonnegative coefficients, and it follows that $P_n$ is 
$\gamma$-positive, as needed.
\end{proof}

\begin{remark}
The argument is analogous to the proof of~\cite[Theorem~1.8]{FMSV}, which hinges on
a recursion for the Hilbert series of the Chow ring of an arbitrary matroid, defined by means
of maximal building sets, that is very similar to~\eqref{eq:recu}. It is tempting to venture that
a similar recursion may hold for Chow rings of some matroids w.r.t.~more general building 
sets (but simple examples show that $\gamma$-positivity need not hold for arbitrary building
sets). This would immediately imply $\gamma$-positivity for the corresponding Hilbert 
series. Theorem~\ref{thm:gpag} would be recovered as the particular case given by the braid 
matroid with respect to the minimal building set, cf.~\cite[\S5.2]{FMSV}. 
\qede\end{remark}

The first several polynomial $G_n(t):=\gamma(P_n)$ for $n\ge 3$ are
\begin{align*}
&1 \\
&1 \\
&1 + 3 t \\
&1 + 13 t \\
&1 + 38 t + 45 t^2 \\
&1+ 94 t + 423 t^2 \\
&1 + 213 t + 2425 t^2 + 1575 t^3 \\
&1 + 459 t + 11017 t^2 + 25497 t^3 \\
&1 + 960 t + 43768 t^2 + 240066 t^3 + 99225 t^4\saf. 
\end{align*}
It would be interesting to study these polynomials further. The polynomial $P_n$ is
real-rooted {\em if and only if\/} $G_n$ is real-rooted (\cite[Proposition~2.9]{FMSV}),
so in order to prove Conjecture~\ref{conj:rr}, it would suffice to prove that $G_n$ is
real-rooted for all $n\ge 3$. 

Using the recursion~\eqref{eq:recuga}, it is easy to show that the formal power series
\[
G(z):=\sum_{n\ge 3} G_n \frac{z^{n-1}}{(n-1)!}
\]
is the unique solution of the differential equation
\[
\frac{dG}{dz}=\frac{z+G}{1-tG}
\]
satisfying $G(0)=0$.

%%%

\section{The coefficient of $\Lbb^k$ in $[\ocaM_{0,n}]$}\label{sec:difeq}

Keel's recursion (\cite[p.~550]{MR1034665}) relates the betti numbers $a_{k,n}$ of 
$\ocaM_{0,n}$ to the numbers~$a_{\ell,m}$ for all $0\le \ell\le k$, $3\le m< n$.
This does not suffice for investigating log-concavity, since we need specific information 
about $a_{k,n}$ for individual $k$. In this section we obtain a precise description of
the corresponding generating functions, from which we will extract in~\S\ref{sec:proof} 
the asymptotic behavior of $a_{k,n}=\rk H^{2k}(\ocaM_{0,n})$ for fixed $k$, as $n\to \infty$.
Our result below appears to be new in this form, notwithstanding the very extensive 
literature on the cohomology of $\ocaM_{0,n}$.

As in the introduction, set 
\[
\alpha_k(z)=\sum_{n\ge 3} a_{k,n}\frac{z^{n-1}}{(n-1)!}
=\sum_{n\ge 3} \rk H^{2k}(\ocaM_{0,n})\frac{z^{n-1}}{(n-1)!} \saf,
\]
a generating function for the coefficients of $\Lbb^k$ in $[\ocaM_{0,n}]$.

\begin{theorem}\label{thm:main2sharp}
We have $\alpha_0(z)=e^z-(z+1)$. For all $k>0$,
\begin{equation}\label{eq:alphak}
\alpha_k(z) = \frac{(k+1)^k}{(k+1)!}\, e^{(k+1)z} 
+ e^z\sum_{m=1}^k (-1)^m p_m^{(k)}(z)\, e^{(k-m) z}
\end{equation}
where $p_m^{(k)}(z)\in \Qbb[z]$, $1\le m\le k$, is a polynomial of degree $2m$ with 
positive leading coefficient.
\end{theorem}

\begin{proof}
Since $\rk H^0(\ocaM_{0,n})=1$ for all $n\ge 3$, we have 
$\alpha_0(z)=\sum_{n\ge 3}\frac{z^{n-1}}{(n-1)!}=e^z-(1+z)$ as stated.
Next, consider the formal power series
\[
M(z):=\sum_{n\ge 3} [\ocaM_{0,n}]\frac{z^{n-1}}{(n-1)!}
\]
with coefficients in $K(\Var_k)$. The recursion \eqref{eq:recu} implies easily that $M$
satisfies the following differential equation:
\begin{equation}\label{eq:difeq}
\frac{dM}{dz}=\frac{z+(1+\Lbb) M}{1- \Lbb M}\saf.
\end{equation}
({\em Mutatis mutandis,\/} this is equivalent to~\cite[(0.8)]{MR1363064}.)
The function $\alpha_k(z)$ are the coefficients of $M$ as a power series in $\Lbb$:
\[
M=\alpha_0(z)+\alpha_1(z)\Lbb +\alpha_2(z)\Lbb^2 +\cdots\saf.
\]
Imposing that this series satisfies~\eqref{eq:difeq} and reading off the coefficients
of $\Lbb^k$ gives us
differential equations for these coefficients. The first few such equations are
\begin{align*}
\frac{d\alpha_0}{dz}&= \alpha_0+z \\
\frac{d\alpha_1}{dz}&= \alpha_0^2 + \alpha_0z + \alpha_0 + \alpha_1 \\
\frac{d\alpha_2}{dz}&= \alpha_0^3 + \alpha_0^2 z + \alpha_0^2 + 2 \alpha_0 \alpha_1 + 
\alpha_1 z + \alpha_1 + \alpha_2 \\
\frac{d\alpha_3}{dz}&= \alpha_0^4 +  \alpha_0^3 z +  \alpha_0^3 + 3  \alpha_0^2  \alpha_1 + 
2  \alpha_0  \alpha_1 z + 2  \alpha_0  \alpha_1 + 2  \alpha_0  \alpha_2 +  \alpha_1^2 
+  \alpha_2 z +  \alpha_2 +  \alpha_3 \\
&\cdots
\end{align*}
and solving them recursively, they take the form
\begin{align*}
\frac{d\alpha_0}{dz}&= \alpha_0+z \\
\frac{d\alpha_1}{dz}&= \alpha_1+e^{2z}-e^zz-e^z \\
\frac{d\alpha_2}{dz}&= \alpha_2+3 e^{3z} - (z^2+5z+5) e^{2z} + \frac{z^3+5z^2+8z+4}2 e^z \\
&\cdots.
\end{align*}

The theorem will be an easy consequence of the following result.

\begin{lemma}\label{lem:difeqform}
For $k\ge 1$, the function $\alpha_k(z)$ satisfies a differential equation of the form
\begin{equation}\label{eq:gfdfindu}
\frac{d\alpha_k}{dz} = \alpha_k+k\,\frac{(k+1)^k}{(k+1)!}\,
e^{(k+1)z} +e^z\sum_{m=1}^k (-1)^m f_m^{(k)}(z)\, e^{(k-m) z}
\end{equation}
where $f_m^{(k)}(z)\in\Qbb[z]$ denotes a polynomial of degree $2m$ with positive 
leading coefficient for $m=1,\cdots,k-1$, and $f_k^{(k)}(z)\in \Qbb[z]$ is a polynomial
of degree $2k-1$ with positive leading coefficient.
\end{lemma}

To see that~\eqref{eq:gfdfindu} implies~\eqref{eq:alphak}, set $\alpha_k=e^z A_k$; by~\eqref{eq:gfdfindu},
\[
\frac{dA_k}{dz} = k\,\frac{(k+1)^k}{(k+1)!}\,
e^{kz}+\sum_{m=1}^k (-1)^m f_m^{(k)}(z) e^{(k-m) z}\saf,
\]
from which
\[
A_k=\frac{(k+1)^k}{(k+1)!}\,e^{kz}+\sum_{m=1}^k (-1)^m p_m^{(k)}(z) e^{(k-m)z}
\]
where $p_m^{(k)}(z)\in \Qbb[z]$ are determined by integration by parts and we
absorb the constant of integration in the summation. For $k-m>0$, 
$\deg p_m^{(k)}=\deg f_m^{(k)}=2m$; for $m=k$, $\deg p_m^{(k)}
=1+\deg f_k^{(k)}=2k$. The leading coefficient of $p_m^{(k)}(z)$ has the same 
sign as the leading coefficient of $f_m^{(k)}$. 
The expression~\eqref{eq:alphak} for $\alpha_k=e^z A_k$ given in 
Theorem~\ref{thm:main2sharp} follows.\smallskip

Therefore, we only need to prove Lemma~\ref{lem:difeqform}.

\begin{proof}[Proof of Lemma~\ref{lem:difeqform}]
For all $i>0$, consider the two statements
\begin{gather}
\tag{$L_i$} 
\frac{d\alpha_i}{dz} = \alpha_i+ e^z\sum_{m=0}^i (-1)^m f_m^{(k)}(z)\, e^{(k-m) z} \\
\tag{$T_i$}
\alpha_i(z) = e^z\sum_{m=0}^i (-1)^m p_m^{(k)}(z)\, e^{(i-m) z}
\end{gather}
where $f_0^{(k)}(z)=k\,\frac{(k+1)^k}{(k+1)!}$, $p_0^{(k)}(z)=\frac{(k+1)^k}{(k+1)!}$,
and the other polynomials $f_m(z)$, $p_m(z)$ satisfy the conditions listed
in Lemma~\ref{lem:difeqform} and Theorem~\ref{thm:main2sharp}.

We have to prove that $(L_k)$ holds for all $k>0$. As shown above, $(L_1)$
and $(L_2)$ hold. We work by strong induction.
By the argument preceding this proof, $(L_i)\implies (T_i)$. Therefore, in proving
$(L_k)$ we may assume the truth of both~$(L_i)$ and $(T_i)$ for all $1\le i<k$,
as well as the expression $\alpha_0(z)=e^z-(z+1)$, which we have already 
verified.

Rewrite~\eqref{eq:difeq} as
\begin{align*}
\frac{dM}{dz} &= \frac{z + (1+\Lbb) M}{1-\Lbb M}
=(z+M)+\Lbb M\frac{1+z+M}{1-\Lbb M} \\
&= z+ M + \sum_{\ell\ge 1} \Lbb^\ell M^\ell (1+z+M)\saf.
\end{align*}
The equation satisfied by the coefficient of $\Lbb^k$ for $k>0$ is
\begin{equation}\label{eq:preq}
\frac{d\alpha_k}{dz} = \alpha_k+\sum_{\ell=1}^k \text{coefficient of $\Lbb^{k-\ell}$ in }
M^\ell (1+z+M)\saf.
\end{equation}
The coefficients of $\Lbb^{k-\ell}$ in $M^\ell$ and $M^{\ell+1}$ are respectively
\begin{equation}\label{eq:twoterms}
\sum_{i_1+\cdots + i_\ell=k-\ell} \alpha_{i_1}\cdots \alpha_{i_\ell}\saf;
\quad
\sum_{i_1+\cdots + i_{\ell+1}=k-\ell} \alpha_{i_1}\cdots \alpha_{i_{\ell+1}}\saf.
\end{equation}
Since $\ell\ge 1$, the expressions only involve terms $\alpha_i$ with $i<k$, 
which by induction may be assumed to satisfy $(T_i)$ for $i>0$ and equal
$e^z-(z+1)$ for $i=0$.

It is clear (by induction) that the summation in~\eqref{eq:preq} is a linear combination of
exponentials with coefficients in $\Qbb[z]$. In order to prove~$(L_k)$, 
we have to prove the following.

\begin{claim}\label{claim:industep}
With notation as above, the coefficient of~$e^{rz}$ in
\begin{equation}\label{eq:bigsum}
\sum_{\ell=1}^k \left((1+z) \sum_{i_1+\cdots+i_\ell=k-\ell} \alpha_{i_1}\cdots \alpha_{i_\ell}
+\sum_{i_1+\cdots+i_{\ell+1}=k-\ell} \alpha_{i_1}\cdots \alpha_{i_{\ell+1}}\right)
\end{equation}
equals 
\[
\begin{cases}
0\quad & \text{if $r<1$ or $r>k+1$;} \\
\text{a polynomial of degree $2k-1$ and sign of l.c.~$(-1)^k$}\quad &
\text{if $r=1$;} \\
\text{a polynomial of degree $2(k+1-r)$ and sign of l.c.~$(-1)^{k+1-r}$}\quad &
\text{if $1< r\le k+1$}\saf.
\end{cases}
\]
Further, the coefficient of the dominant term $e^{(k+1)z}$ equals 
$f_0^{(k)}=k\,\dfrac{(k+1)^k}{(k+1)!}$.
\end{claim}

Since our main application (to asymptotic log-concavity) concerns the dominant 
term, we focus on the last statement in Claim~\ref{claim:industep} first.
For this, note that by the induction hypothesis the two terms in~\eqref{eq:twoterms} 
respectively equal
\[
\left(\sum_{i_1+\cdots + i_\ell=k-\ell} \prod_{m=1}^\ell \frac{(i_m+1)^{i_m}}
{(i_m+1)!}\right) e^{kz} + \text{lower order terms}
\]
and
\[
\left(\sum_{i_1+\cdots + i_{\ell+1}=k-\ell} \prod_{m=1}^{\ell+1} \frac{(i_m+1)^{i_m}}
{(i_m+1)!}\right) e^{(k+1)z} + \text{lower order terms}
\]
where `lower order terms' stands for a linear combination with polynomial coefficients
of exponentials $e^{m z}$ with $m<k$, resp., $m<k+1$. Therefore, the equation satisfied
by $\alpha_k$ is
\begin{equation}\label{eq:difeqMk}
\frac{d\alpha_k}{dz} = \alpha_k+\sum_{\ell=1}^k\left(\sum_{i_1+\cdots + i_{\ell+1}=k-\ell} \prod_{m=1}^{\ell+1}
\frac{(i_m+1)^{i_m}}{(i_m+1)!}\right) e^{(k+1)z}+\text{lower order terms.}
\end{equation}

\begin{lemma}\label{lem:Lamb}
\[
\sum_{\ell=1}^k\left(\sum_{j_1+\cdots + j_{\ell+1}=k-\ell} \prod_{m=1}^{\ell+1}
\frac{(j_m+1)^{j_m}}{(j_m+1)!}\right) =k\frac{(k+1)^k}{(k+1)!}\saf.
\]
\end{lemma}

\begin{proof}
Let
\[
W(t):=\sum_{j\ge 0} \frac{(-(j+1))^j}{(j+1)!} t^{j+1}\saf.
\]
This is the {\em principal branch of the Lambert W function;\/} in particular,
\[
W(t) e^{W(t)}=t
\]
(see e.g., \cite[(3.1)]{MR1414285}). By implicit differentiation,
\[
\frac{dW}{dt}=\frac{W(t)}{t(1+W(t))}\saf,
\]
and it follows that
\begin{equation}\label{eq:Lambert}
t^2\frac d{dt}\left(\frac{W(t)}t\right)=-\frac{W(t)^2}{1+W(t)} = -W(t)^2+W(t)^3-W(t)^4+\cdots
\end{equation}
The coefficient of $t^{k+1}$ in the l.h.s.~of~\eqref{eq:Lambert} is
\[
(-1)^k\, k\,\frac{(k+1)^k}{(k+1)!}\saf.
\]
The coefficient of $t^{k+1}$ in the r.h.s.~equals the coefficient of $t^{k+1}$ in
\[
\sum_{\ell=1}^k (-1)^\ell W(t)^{\ell+1}\saf.
\]
Now $(j_1+1)+\cdots+(j_{\ell+1}+1)=k+1$ if and only if $j_1+\cdots+j_{\ell+1}=k-\ell$, therefore 
the coefficient of $t^{k+1}$ in $(-1)^\ell W(t)^{\ell+1}$ equals
\[
(-1)^\ell\sum_{j_1+\cdots+j_{\ell+1}=k-\ell} \prod_{m=1}^{\ell+1}
\frac{-(j_m+1)^{j_m}}{(j_m+1)!}
=(-1)^k\sum_{j_1+\cdots+j_{\ell+1}=k-\ell} \prod_{m=1}^{\ell+1}
\frac{(j_m+1)^{j_m}}{(j_m+1)!}
\]
and this concludes the proof.
\end{proof}

By Lemma~\ref{lem:Lamb}, we can rewrite~\eqref{eq:difeqMk} as
\[
\frac{d\alpha_k}{dz} = \alpha_k+k\,\frac{(k+1)^k}{(k+1)!}\,
e^{(k+1)z}+\text{lower order terms}
\]
and this concludes the verification that $f_0^{(k)}=k\,\dfrac{(k+1)^k}{(k+1)!}$
as stated in Claim~\ref{claim:industep}.

The rest of the proof of~Claim~\ref{claim:industep} is a straightforward, but 
somewhat involved, verification. If all $i_j$ are $<k$ and positive, then by the induction
hypothesis
\[
\alpha_{i_1}\cdots \alpha_{i_s}=e^{sz} \sum_{m=0}^{i_1+\cdots+i_s}(-1)^m g_m(z)
e^{(i_1+\cdots+i_s-m)z}
\]
with $g_m(z)\in \Qbb[z]$ a polynomial of degree $2m$ and positive leading
coefficient. The coefficient of $e^{rz}$ in this term is
\begin{equation}\label{eq:allnz}
\begin{cases}
0 \quad & \text{if $r<s$ or $r>i_1+\cdots+i_s+s$} \\
(-1)^{i_1+\cdots+i_s+s-r} g_{i_1+\cdots+i_s-(r-s)}(z) \quad & \text{if $s\le r\le i_1+\cdots+i_s+s$.}
\end{cases}
\end{equation}
On the other hand,
\[
\alpha_0^t = (e^z-(z+1))^t = \sum_{m=0}^t \binom tm (-1)^{t-m} (z+1)^{t-m} e^{mz}\saf,
\]
therefore the coefficient of $e^{rz}$ in $\alpha_0^t$ is
\begin{equation}\label{eq:allz}
\begin{cases}
0 \quad & \text{if $r<0$ or $r>t$} \\
\binom tr (-1)^{t-r} (z+1)^{t-r} \quad & \text{if $0\le r\le t$.}
\end{cases}
\end{equation}

We will frequently refer to~\eqref{eq:allnz} and~\eqref{eq:allz} in the rest of the proof.\smallskip

First, \eqref{eq:allnz} and~\eqref{eq:allz} imply that the coefficient of $e^{rz}$ 
in~\eqref{eq:bigsum} is possibly nonzero only if $0\le r\le k+1$. Indeed, the maximum 
exponent for $\alpha_0^t \alpha_{i_1}\cdots \alpha_{i_s}$, where all $i_j$ are positive, is
\[
t+ (i_1+\cdots+i_s)+s
\]
by~\eqref{eq:allnz} and~\eqref{eq:allz},
so for $t+s=\ell+1$ and $i_1+\cdots+i_s=\ell-k$ it equals $k+1$.\smallskip

Next, consider the case $r=0$,
that is, the term in~\eqref{eq:bigsum} not involving exponentials. By~\eqref{eq:allnz},
the only possibly nonzero contributions to $r=0$ in~\eqref{eq:bigsum} come from 
terms with all~$i_j$ equal to $0$. However, in this case $\sum i_j=0$, that is, $k-\ell=0$,
and the corresponding summands in~\eqref{eq:bigsum} are
\begin{equation}\label{eq:leqk}
(1+z) \alpha_0^k + \alpha_0^{k+1} = \alpha_0^k (1+z+\alpha_0)=(e^z-(z+1))^k e^z\saf.
\end{equation}
This is a multiple of $e^z$, therefore the contribution to $r=0$ vanishes, as
stated in Claim~\ref{claim:industep}.\smallskip

For $r=1$: By~\eqref{eq:allnz}, at most one index may be nonzero. If all indices equal~$0$,
then $\ell=k$ as in the previous case, the corresponding part of~\eqref{eq:bigsum} 
is~\eqref{eq:leqk}, and the coefficient of $e^z$ equals $(-1)^k(1+z)^k$.
For $1\le \ell< k$, the corresponding contribution to~\eqref{eq:bigsum} is
\[
\ell (1+z) \alpha_0^{\ell-1} \alpha_{k-\ell}+(\ell+1) \alpha_0^\ell \alpha_{k-\ell}
=\alpha_0^{\ell-1} \alpha_{k-\ell} \left((\ell+1)e^z-(1+z)\right)\saf.
\]
Now $\alpha_{k-\ell}$ is a multiple of $e^z$, so the coefficient of $e^z$ in this expression
is the coefficient in $-\alpha_0^{\ell-1} \alpha_{k-\ell} (1+z)$, that is,
\[
(-1)^k (1+z)^\ell p_{k-\ell}^{(k-\ell)}(z)\saf.
\]
These polynomials have degrees $k+1,\dots, 2k-1$ as $\ell=k-1,\dots, 1$.

The conclusion is that the coefficient of $e^z$ in~\eqref{eq:bigsum} has degree $2k-1$
and sign of leading coefficient $(-1)^k$, as stated in~Claim~\ref{claim:industep}.\smallskip

Finally, we consider the case $1<r\le k+1$. Each $\alpha_i$ with $i>0$ is a multiple
of $e^z$, so the coefficient of $e^{rz}$ is nonzero only for terms in~\eqref{eq:bigsum} 
with at most $r$ indices $i_j>1$. These terms are of two types. 
First, we have terms
\begin{equation}\label{eq:firan}
(1+z)\alpha_0^t \alpha_{i_1}\cdots \alpha_{i_s}
\end{equation}
with $s\le r$, all $i_j$ positive, $i_1+\cdots +i_s=k-\ell$, and $s+t=\ell$.
The maximum $r$ for which $e^{rz}$ appears in~\eqref{eq:firan} is
\[
t+(i_1+1)+\cdots+(i_s+1)=s+t + \sum i_j = \ell+k-\ell = k\saf.
\]
Therefore~\eqref{eq:firan} does not contribute to the coefficient of $e^{rz}$ if $r=k+1$.

For $1<r\le k$, \eqref{eq:allnz} and~\eqref{eq:allz} imply that the coefficient of $e^{rz}$
in~\eqref{eq:firan} equals
\[
\sum_{r_1+r_2=r} \binom t{r_1} (-1)^{k-r} (z+1)^{t-r_1+1}
g_{k-r-(t-r_1)}(z)\saf.
\]
For an individual summand to be nonzero we need $r_1\le t$, i.e., $r_1-t\le 0$, as well as 
$s\le r_2$, i.e., $r_1-t\le r-\ell$.
Each nonzero summand has degree
\[
t-r_1+1+2(k-r-(t-r_1))=2k-2r+1+r_1-t\saf;
\]
since $r_1-t\le \min(0,r-\ell)$ for nonzero summands, the maximum of this expression 
is
\[
2k-2r+1+\min(0,r-\ell)< 2(k-r+1)\saf.
\]
Therefore, if $1<r\le k$, the coefficient of $e^{rz}$ in each term of type~\eqref{eq:firan} 
is a polynomial of degree strictly less than $2(k-r+1)$.

The other possible type is
\begin{equation}\label{eq:secan}
\alpha_0^t \alpha_{i_1}\cdots \alpha_{i_s}
\end{equation}
with $s\le r$, all $i_j$ positive, $i_1+\cdots +i_s=k-\ell$, and $s+t=\ell+1$.
By~\eqref{eq:allnz} and~\eqref{eq:allz}, the coefficient of~$e^{rz}$ in this term equals
\[
\sum_{r_1+r_2=r} \binom t{r_1} (-1)^{k-r+1} (z+1)^{t-r_1}
g_{k+1-r-(t-r_1)}\saf.
\]
We argue as above: nonzero individual summands have $r_1-t\le 0$
and $s\le r_2$, i.e., $r_1-t\le r-(\ell+1)$. Now note that as $1<r$, for all $r$ the 
sum~\eqref{eq:bigsum} will include terms~\eqref{eq:secan} with $\ell+1\le r$.
For these terms, the condition $r_1-t\le 0$ implies the condition $s\le r_2$;
the degree of the summand,
\[
(t-r_1)+2(k+1-r-(t-r_1))=2(k+1-r)+(r_1-t)\saf,
\]
achieves its maximum for $r_1=t$ and equals $2(k+1-r)$. All these summands
are of the form
\[
(-1)^{k-r+1} g_{k+1-r}\saf,
\]
so the sign of their leading coefficient is $(-1)^{k+1-r}$.\smallskip

We conclude that, for $1<r\le k$, the coefficient of $e^{rz}$ in~\eqref{eq:bigsum}
is the sum of polynomials of degree $<2(k+1-r)$ obtained from terms of 
type~\eqref{eq:firan} and from terms of type~\eqref{eq:secan} with $\ell\ge r$,
and of polynomials of degree exactly $2(k+1-r)$ and sign of leading 
coefficient~$(-1)^{k+1-r}$, from terms of type~\eqref{eq:secan} with $\ell<r$.

Therefore in this case the coefficient of $e^{rz}$ in~\eqref{eq:bigsum} is a
polynomial of degree~$2(k+1-r)$ and sign of leading coefficient
$(-1)^{k+1-r}$, and this completes the verification of Claim~\ref{claim:industep}.
\end{proof}

This concludes the proof of Lemma~\ref{lem:difeqform} and therefore
of Theorem~\ref{thm:main2sharp}.
\end{proof}

Theorem~\ref{thm:main2sharp} identifies the degrees and signs of leading coefficients
of the coefficients~$p_m^{(k)}(z)$, $m=1,\dots,k$, in the expression
\[
\sum_{n\ge 3} \rk H^{2k}(\ocaM_{0,n})\frac{z^{n-1}}{(n-1)!} 
=\frac{(k+1)^k}{(k+1)!}\, e^{(k+1)z} 
+ e^z\sum_{m=1}^k (-1)^m p_m^{(k)}(z)\, e^{(k-m) z}\saf,
\]
valid for $k\ge 1$. The first several such coefficients are
\begin{align*}
p_1^{(1)} &=\frac 12 z^2 + z + 1 \\[10pt]
p_2^{(1)} &=z^2 + 3z + 2 \\
p_2^{(2)} &=\frac 18 z^4 + \frac 56 z^3 + 2 z^2 + 2 z + \frac 12 \\[10pt]
p_3^{(1)} &=\frac 94 z^2 + \frac{15}2 z + 5 \\
p_3^{(2)} &=\frac 12 z^4 + \frac{11}3 z^3 + 9 z^2 + 9z + 3 \\
p_3^{(3)} &=\frac 1{48} z^6 + \frac 7{24} z^5 + \frac{35}{24} z^4 + \frac 72 z^3 
+ \frac{17}4 z^2 + \frac 52 z + \frac 23 \\[10pt]
p_4^{(1)} &=\frac{16}3 z^2 + \frac{56}3 z + \frac{38}3 \\
p_4^{(2)} &=\frac{27}{16} z^4 + \frac{51}4 z^3 + \frac{129}4 z^2 + \frac{65}2 z 
+ \frac{45}4 \\
p_4^{(3)} &=\frac 16 z^6 + \frac{13}6 z^5 + \frac{21}2 z^4 + \frac{74}3 z^3 
+ 30 z^2 + 18 z + \frac{13}3 \\
p_4^{(4)} &=\frac 1{384} z^8 + \frac 1{16} z^7 + \frac 59 z^6 + \frac {49}{20} z^5 
+ \frac{289}{48} z^4 + \frac{103}{12} z^3 + \frac {85}{12} z^2 + \frac {19}6 z + \frac{13}{24} 
\end{align*}
The leading coefficients of these polynomials are positive by Theorem~\ref{thm:main2sharp},
but note that the polynomials themselves appear to be positive. In fact, numerical evidence 
suggests the following.

\begin{conj}\label{conj:posp}
For all $k\ge 1$, the polynomials $p_m^{(k)}$, $m=1,\dots,k$, have positive coefficients
and are log-concave with no internal zeros. All but $p^{(1)}_1$, $p^{(3)}_3$, 
$p^{(5)}_5$ are ultra-log-concave.
\end{conj}

One could attempt to determine generating functions for these coefficients, by a more
careful analysis of the induction proving Lemma~\ref{lem:difeqform}. For instance,
the following should hold:
\[
\sum_{k\ge 1} p_k^{(k)} t^k = \frac 1{e^z} \frac 1{(1+t)^{1/t}(1-t(z+1))^{1/t}}\saf.
\]

%%%

\section{Proof of Theorems~\ref{thm:main} and~\ref{thm:main2}}\label{sec:proof}

Theorem~\ref{thm:main2} follows easily from Theorem~\ref{thm:main2sharp}.
In fact, Theorem~\ref{thm:main2sharp} implies the following more precise statement.
Denote by $c_{mj}^{(k)}\in \Qbb$ the coefficients of $p_m^{(k)}(z)$:
\[
p_m^{(k)}(z)=\sum_{j=0}^{2m} c_{mj}^{(k)} z^j\saf.
\]

\begin{theorem}\label{thm:main2again}
Let $n\ge 3$. For every $k\ge 1$:
\[
a_{k,n} = \rk H^{2k}(\ocaM_{0,n})=\frac{(k+1)^{k+n-1}}{(k+1)!} 
+\sum_{m=1}^k (-1)^m\sum_{j=0}^{2m} \binom{n-1}j\, c_{mj}^{(k)}\, j!
(k-m+1)^{n-1-j}\saf.
\]
\end{theorem}

\begin{proof}
By definition, $a_{k,n}$ is the coefficient of $\frac{z^{n-1}}{(n-1)!}$ in the expansion of 
$\alpha_k(z)$. The stated formula follows from Theorem~\ref{thm:main2sharp}.
\end{proof}

\begin{remark}\label{rem:clo}
By Theorem~\ref{thm:main2sharp},
$c_{m,2m}^{(k)}>0$. By Theorem~\ref{thm:main2again}, 
\[
\rk H^{2k}(\ocaM_{0,n})=\frac{(k+1)^{k-1}}{(k+1)!}\cdot (k+1)^n
+\sum_{m=1}^k (-1)^m q^{(k)}_m(n)\cdot (k+1-m)^n
\]
where
\[
q^{(k)}_m(n)=\sum_{j=0}^{2m} \frac{c_{mj}^{(k)}}{(k-m+1)^{j+1}} (n-1)\cdots (n-j)
\]
is a polynomial in $\Qbb[n]$ of degree $2m$ and with positive leading coefficient.

If the positivity claim in Conjecture~\ref{conj:posp} holds, then all coefficients $c_{mj}^{(k)}$ 
are positive for $k\ge 1$, $1\le m\le k$, $0\le j\le 2m$. 
\qede\end{remark}

\begin{example}
We have:
\begin{align*}
\rk H^2(\ocaM_{0,n}) &= \frac 12\cdot 2^n - \frac{n^2 - n + 2}2	
\quad\qquad\qquad\qquad\qquad\qquad\qquad\qquad
\text{(cf.~\cite[p.~550]{MR1034665})}  \\
\rk H^4(\ocaM_{0,n}) &=\frac 12\cdot 3^n - \frac{n^2 + 3n + 4}8\cdot 2^n 
+ \frac{3n^4 - 10n^3 + 33 n^2 - 26 n + 12}{24}
\end{align*} 
and
\begin{multline*}
\rk H^6(\ocaM_{0,n})=
\frac 23\cdot 4^n
-\frac{(n+4)(n+3)}{12}\cdot 3^n
+\frac{3n^4+14n^3+57n^2+118n+96}{192}\cdot 2^n \\
-\frac{n^6-7n^5+35n^4-77n^3+120n^2-72n+32}{48}\saf.
\end{multline*}
\end{example}

\begin{proof}[Proof of Theorem~\ref{thm:main2}]
The statement is true for $k=0$. For $k>0$, it is an immediate consequence of 
Theorem~\ref{thm:main2again}, since (with notation as in Remark~\ref{rem:clo})
\[
\lim_{n\to \infty} q_m^{(k)}(n) \frac{(k-m+1)^n}{(k+1)^n} =0
\]
for $1\le m\le k$.
\end{proof}

Finally, we deduce Theorem~\ref{thm:main} from Theorem~\ref{thm:main2}.

\begin{proof}[Proof of Theorem~\ref{thm:main}]
We verify the stronger claim that for any fixed $k>0$, the limit of the ratio
\begin{equation}\label{eq:ratio}
\left(\frac{a_{k,n}}{\binom{n-3}k}\right)^2\left/
\left(\frac{a_{k-1,n}}{\binom{n-3}{k-1}}\cdot \frac{a_{k+1,n}}{\binom{n-3}{k+1}}\right)\right.
\end{equation}
as $n\to \infty$ is $+\infty$. The terms involved in the ratio are of type
\[
a_{i,n}\left/\binom{n-3}i\right.\saf,
\]
and by Theorem~\ref{thm:main2} this is asymptotic to
\[
\frac{(i+1)^{i+n-1}}{(i+1)!}\left/\frac{(n-3)!}{i! (n-i-3)!}\right.
=\frac{i!(i+1)^{i+n-1} (n-i-3)!}{(i+1)!(n-3)!}
=\frac{(i+1)^{i+n-2} (n-i-3)!}{(n-3)!}\saf.
\]
Thus, the limit of~\eqref{eq:ratio} as $n\to \infty$ equals the limit of
\[
\left(\frac{(k+1)^{k+n-2} (n-k-3)!}{(n-3)!}\right)^2\left/
\left(\frac{k^{k+n-3} (n-k-2)!}{(n-3)!}\frac{(k+2)^{k+n-1} (n-k-4)!}{(n-3)!}
\right)\right.
\]
as $n\to \infty$. This expression equals
\[
\frac{(k+1)^{2(k+n-2)} (n-k-3)}{k^{k+n-3} (k+2)^{k+n-1} (n-k-2)}
=\frac{(k+1)^{2(k-2)}}{k^{k-3} (k+2)^{k-1} }
\cdot \left(\frac{(k+1)^{2}}{k (k+2)}\right)^n
\cdot \frac{n-k-3}{n-k-2}
\]
and converges to $\infty$ as claimed as $n\to \infty$, since $\frac{(k+1)^{2}}{k (k+2)}
=\frac{k^2+2k+1}{k^2+2k\hphantom{+1}}>1$.
\end{proof}

%%%

\newcommand{\etalchar}[1]{$^{#1}$}

\end{document}